\documentclass[12pt]{amsart}
\usepackage{amssymb,amscd,amsmath}
\usepackage[makeroom]{cancel}
\newtheorem{theorem}{Theorem}

\newtheorem{example}[theorem]{Example}
\newtheorem{remark}[theorem]{Remark}
\newtheorem{proposition}[theorem]{Proposition}

\def\vt{\vartheta}

\def\ZZ{{\mathbb Z}}
\def\CC{{\mathbb C}}

\DeclareMathOperator{\GL}{GL}  
\DeclareMathOperator{\SL}{SL}
\DeclareMathOperator{\SO}{SO}
\DeclareMathOperator{\Sp}{Sp}
\def\Em{{\mathfrak E} }%
\def\EE{{\mathbb E}}%
\def\T{{\mathbb T}}
\def\tt{{\mathfrak t}}
\def\LL#1{{\mathcal L({#1})}}
\def\GL{{\rm GL}}
\def\nub{{\bar \nu}}
\def\mub{{\bar \mu}}
\def\zetab{{\bar \zeta}}
\def\z{{\bar z}}
\def\s{{s^*}}
\def\id{{\rm id}}
\DeclareMathOperator\pt{pt}

\def\vee{{\sf v}}
\def\om{\omega}
\def\tao{\tau_0}
\def\cc{{\rm c}}
\def\e{{\rm e}}
\def\ii{{\mathfrak i}}
\def\Tf{{\mathcal T}}
\def\Ff{{\mathcal F}}
\def\Pp{\Pi}
\def\Phh{\Phi}

\usepackage[dvipsnames]{xcolor}

\title{Elliptic classes on Langlands dual flag varieties}

\author{Rich\'ard Rim\'anyi}
\address{Department of Mathematics, University of North Carolina at Chapel Hill, USA}
\email{rimanyi@email.unc.edu}

\author{Andrzej Weber}
\address{Institute of Mathematics, University of Warsaw, Poland}
\email{aweber@mimuw.edu.pl}

\begin{document}

\begin{abstract} Characteristic classes of Schubert varieties can be used to study the geometry and the combinatorics of homogeneous spaces. We prove a relation between elliptic classes of Schubert varieties on a generalized full flag variety and those on its Langlands dual. This new symmetry is motivated by 3d mirror symmetry, and it is only revealed if Schubert calculus is elevated from cohomology or K theory to the elliptic level.  
\end{abstract}

\maketitle

\section{Introduction}

\subsection{Elliptic characteristic classes}
Schubert classes are the cohomology fundamental classes of Schubert varieties in homogeneous spaces. They play an important role in the geometric study of those spaces, as well as in algebraic combinatorics. In recent years the theory of characteristic classes of singular varieties developed under the influence of geometry, representation theory, and physics. Now we can regard Schubert classes in {\em torus equivariant elliptic cohomology} instead of ordinary cohomology \cite{RW19, KRW20}. A remarkable new feature of equivariant elliptic Schubert classes $\EE(X_\om)$ is that they necessarily depend on new variables, called dynamical (or K\"ahler) variables, and another one called $h$. Hence, when we restrict a class $\EE(X_\om)$ to a torus fixed point we obtain an elliptic function in equivariant variables, dynamical variables, and $h$. 

To be more concrete, let us assume that our underlying homogeneous space is the generalized full flag variety $G/B$ for a simply connected semisimple linear group $G$. Then the restriction of the class $\EE(X_\om)$ to a fixed point indexed by $\sigma$ (both $\om$ and $\sigma$ are elements of the corresponding Weyl group) is a function of the variables
\begin{equation}\label{introeq}
\begin{tabular}{lcl}
$\zeta_s:\T\to \CC^*$ &  & (equivariant variables),  \\
$\nu_s:\T^\vee\to \CC^*$ & & (dynamical variables),  \\
$h\in \CC$
\end{tabular}
\end{equation}
for $s\in S=\{$simple reflections$\}$.  

\subsection{3d mirror symmetry} \label{sec:physics}
A $d=3$, $\mathcal N=4$ supersymmetric gauge theory has an associated {\em Higgs branch} and an associated {\em Coulomb branch}. The parameters of the Higgs branch include the so-called Fayet-Iliopoulos (FI) parameters and the mass parameters.  Between certain pairs of such theories S-duality is constructed, under which their branches are interchanged. Composing S-duality with the comparison of the two branches of the same theory, we arrive at a relation between the respective Higgs branches of two ``3d-mirror-dual'' theories. By the nature of this duality, for these two Higgs branches the FI and mass parameters are interchanged. 

It is predicted that such a physical duality can be observed in mathematics, using elliptic characteristic classes. Namely, for certain pairs of holomorphic symplectic manifolds (with extra structure) the elliptic classes of relevant subvarieties of one should agree with the elliptic classes of relevant subvarieties of the other, {\em in a certain sense}. The ``certain sense'' should involve the interchanging of equivariant (``mass'') and dynamical (``FI'') parameters.  This prediction is now verified in two cases: for the Grassmannian and its dual in \cite{RSVZ0}, and the self duality of type A (classical) full flag variety in \cite{RSVZ}. The elliptic characteristic classes used  in those works are the so-called ``elliptic stable envelopes'' of \cite{AO16}, see also \cite{FRV18, RTV}.

In this paper we generalize the self duality of the type $A$ full flag variety to arbitrary type. Instead of using the language of stable envelopes (which are formally only defined for Nakajima quiver varieties, and $T^*G/B$ is a quiver variety only in type A), we use the language of elliptic Schubert classes introduced in \cite{RW19, KRW20}. In type A Schubert classes and stable envelopes are the same.

The 3d mirror dual of $T^*G/P$, even in type A, is not expected to be another $T^*G/P$, it is sometimes a quiver variety, and in general a bow variety.

\subsection{The result}  Several geometric and combinatorial aspects of a group $G$ match some related aspects of the Langlands dual group $G^\vee$. 
One example is that their tori are naturally dual tori of each other. Therefore, what is an equivariant (resp. dynamical) parameter for $G$, is a dynamical (resp. equivariant) parameter for $G^\vee$, cf. \eqref{introeq}. 

Let $\EE_\sigma(X_\om)$ be the restriction of $\EE(X_\om)$ to the fixed point $\sigma$, divided by the Euler class $eu(T_\sigma G/B)$ of the tangent bundle.
In this paper we compare the restricted classes
\begin{equation}\label{twofunc}
\EE_{\tao\om^{-1}}(X_{\tao\sigma^{-1}})
\qquad
\text{and}
\qquad
\EE_\sigma(X^\vee_\om). 
\end{equation}
The first one comes from the elliptic cohomology of $G/B$, the second one from that of $G^\vee/B^\vee$ ($X_\om^\vee$ is a Schubert variety in the latter). Here $\tao$ is the longest element of the Weyl group.  Our main result is Theorem \ref{main} below: the two functions in \eqref{twofunc} agree, up to a sign, if we switch the equivariant and dynamical parameters, introduce some reordering and inverses, as well as switch $h$ with $h^{-1}$. Note that 
 $\om$ and $\sigma$ also switch places in the two restricted classes in \eqref{twofunc}.

The proof relies on two recursions, called {\em Bott-Samelson recursion} and {\em R-matrix recursion}, for the elliptic Schubert classes in $G/B$, as well as some Lie theory considerations. 

\bigskip
\noindent {\bf Acknowledgments.} R.R. is supported by Simons Foundation Grant 523882. A.W. is supported by NCN grant 2013/\-08/\-A/\-ST1/\-00804 and 2016/23/G/ST1/04282 (Beethoven 2). The authors are grateful for useful conversations on the topic with A. Okounkov, L. Rozansky, A. Smirnov.

\section{Preliminaries}

Let $G$ be a simply connected semisimple linear group, and $\mathfrak g$ its Lie algebra. Let $\mathfrak g^\vee$ be the Lie algebra Langlands-dual to $\mathfrak g$, that is,
$\mathfrak g$ and $\mathfrak g^\vee$ have dual root systems (see \cite{Langlands}). Let $G^\vee$ be the simply connected semisimple group with Lie algebra $\mathfrak g^\vee$. Denote by $\T$ and $\T^\vee$ the maximal tori in $G$ and $G^\vee$. Their Lie algebras are dual to each other, i.e.
$$Lie(\T)=\tt\,,\quad  Lie(\T^\vee)=\tt^*.$$
The Weyl group $W=N(\T)/\T$ is identified with the subgroup of $\GL(\tt^*)$ generated by the reflections with respect to simple roots. The set of positive roots will be denoted by $\Phh_+$. For $\alpha \in \Phh_+$ let $s_\alpha$ be the corresponding reflection. Conversely, for a reflection $s$ let $\alpha_s$ denote the corresponding positive root. 
The set of all reflections $\Pp\subset W$ is in bijection with the set of positive roots $\Phh_+$. The group is generated by the set of simple reflections $S\subset \Pp$. The roots $\alpha_s$ for $s\in S$ form a basis of $\tt^*$.
The set of positive roots of $G$
$$\Phh_+=\{\alpha_s\}_{s\in \Pp}\subset \tt^*$$
is equal to the set of coroots of $G^\vee$, and the set of coroots of $G$  
$$\Phh^\vee_+=\{\alpha^\vee_s\}_{s\in \Pp}\subset \tt$$
is the set of roots of $G^\vee$.

Define
$$\zeta_s=\e^{-\alpha_s}:\T\to \CC^*\,,\qquad \nu_s=h^{\alpha^\vee_s}=\e^{\epsilon\alpha^\vee_s}:\T^\vee\to\CC^*\,,$$
where $\epsilon$ is a formal variable, 
$h=\e^\epsilon$. We will call $\zeta_s$ the ``equivariant variables'' and $\nu_s$ the ``dynamical variables''.
For $G=\SL_n$ the notation 
$$\zeta_s=\tfrac{z_{s+1}}{z_{s}}\,,\qquad \nu_s=\tfrac{\mu_{s+1}}{\mu_{s}} \qquad (s=1,2,\dots,n-1)$$
was used in \cite{RW19, KRW20}, where $z_s$ and $\mu_s$ are natural variables for $\GL_n$.
Define $$\zetab_s=\nu_s^{-1}:\T^\vee\to\CC^*,\qquad \nub_s=\zeta_s^{-1}:\T\to\CC^*.$$ They play the role of equivariant and dynamical variables for~$G^\vee$.

\smallskip

Consider the left actions of $\T$ and $B$ on $G/B$. Elements of the Weyl group are identified with the $\T$ fixed points according to
$$\iota:\sigma\in W=N(\T)/\T\hookrightarrow G/B.$$
For $\om\in W$ the Schubert variety $X_\om$ is the closure of the $B$ orbit of the point  $\iota(\om)$. We have $X_{\id}=\{\pt\}$ and $X_{\tao}=G/B$, where $\tao$ is the longest element of the Weyl group.

\medskip

For a complex number $q\in \CC^*$ with $|q|<1$ throughout the paper we will use the Jacobi theta function
\[
\vt(x)=\vt(x,q)=x^{1/2}(1-x^{-1})\prod_{n\geq 1}(1-q^nx)(1-q^n/x)\in \ZZ[x^{\pm 1}][[q]],
\]
as well as
\[
\delta(a,b)=\frac{\vt(ab)\vt(1)'}{\vt(a)\vt(b)}=\frac{ab-1}{(a-1)(b-1)}+q(\tfrac1{ab}-ab)+\dots\,.
\]

\section{The elliptic class of a  Schubert variety and its two recursions}

Some of the key works towards elliptic Schubert calculus include the papers \cite{BoLi0, GR13, LZ, AO16}. In \cite{RW19} the elliptic class $E(X_\om)$ is defined for the Schubert variety $X_\om\subset G/B$. As explained in that work, while the class can be set up in different $\T$ equivariant cohomology theories (cohomology, K theory, elliptic cohomology), we choose the version living in K theory:
\[
E(X_\om)\in K_{\T}(G/B)(\nu_i,h)[[q]],
\]
where $\nu_i, h$ are variables as in \eqref{introeq}, and $q\in \CC^*$, $|q|<1$. 

The definition of $E(X_\om)$ follows one of the few strategies of defining characteristic classes of singular varieties: one takes a resolution of $X_\om$, and defines $E(X_\om)$ as the push-forward of an appropriate class from the resolution. The ``appropriate class'' needs to be chosen in such a way that the obtained class on $G/B$ is independent of the resolution. This strategy was pioneered in the elliptic settings by Borisov and Libgober \cite{BoLi0, BoLi1, BoLi2}. In fact in Schubert calculus this strategy only works after a ``twist''  by a line bundle, resulting in a class necessarily depending on the $\nu_i$ and $h$ variables. For more details see \cite{RW19, KRW20}. Post factum, however, there is an alternative approach, via recursions.

All classical characteristic class notions of Schubert varieties (eg. fundamental class in $H^*$, $K$, Chern-Schwartz-MacPherson class, motivic Chern class) satisfy so-called BGG-type recursions (generalized divided difference recursions), see a summary in \cite[Section 5]{RW19}. The elliptic class is not an exception either, for $G/B$ in fact it satisfies two consistent recursions, each of which  can serve as {\em definitions} of $E(X_\om)$ for the purpose of this paper.

\begin{theorem}\cite[Formulas 1.1 and 1.2]{RW19} Let $s\in S$ be a simple reflection and $\alpha_s$ the associated simple root. Let $\LL{\alpha_s}=G\times_B\CC_{-\alpha_s}$ be the line bundle associated with $\alpha_s$. The actions of $W$ on $\tt$, $\tt^*$, $G/B$ are denoted by $s^{\nu}$, $s^{\zeta}$, $s^{\gamma}$, respectively. Then 
\begin{multline}\label{bs-rec}
\delta\left(\LL{\alpha_s},h^{\alpha_s^\vee}\right) 
\cdot s^\nu E(X_\om) -
\delta\left(\LL{\alpha_s},h\right)\cdot
s^\gamma s^\nu E(X_\om)=\\ 
=\begin{cases}E(X_{\om s})& \text{\rm if } \ell(\om s)=\ell(\om)+1\\ \\
\delta(h^{\alpha_s^\vee},h)\delta(h^{-\alpha_s^\vee},h)\cdot E(X_{\om s})
&\text{\rm if }\ell(\om s)=\ell(\om)-1,\end{cases}
\end{multline}
\begin{multline}\label{R-rec}
\delta\left(\e^{-\alpha_s},h^{\om^{-1}\alpha_s^\vee}\right) 
\cdot E(X_\om) - 
\delta\left(\e^{\alpha_s},h\right)\cdot s^\zeta
E(X_\om)=\\
=\begin{cases}E(X_{s\om})& \text{\rm if } \ell(s\om)=\ell(\om)+1\\ \\
\delta(h^{\om^{-1}\alpha_s^\vee},h)\delta(h^{-\om^{-1}\alpha_s^\vee},h)\cdot
E(X_{s\om})&\text{\rm if }\ell(s\om)=\ell(\om)-1,\end{cases}
\end{multline} 
and
$$E(X_\id)=\iota_*(1_{X_\id})\,,$$
where $\iota:X_\id\to G/B$ is the inclusion. 
\end{theorem}
Following \cite{RW19} we call \eqref{bs-rec} the {\em Bott-Samelson recursion}, and \eqref{R-rec} the {\em R-matrix recursion}.
We will need the local form of both recursions. Let
\[
E_{\sigma}(X_\om)=\frac{E(X_\om)|_{\sigma}}{eu(T_\sigma G/B)},
\]
be the restriction of the elliptic class to the central point $\sigma\in X_\sigma$, divided by the Euler classes of the tangent space.
The formulas for local recursions are given in \cite[\S11.2]{RW19}:
\begin{itemize}
\item The Bott-Samelson recursion is
\begin{multline}
\delta\left(\sigma(\zeta_s),\nu_s\right) 
\cdot s^\nu E_\sigma(X_\om) + 
\delta\left(\sigma(\zeta_s),h\right)\cdot
s^\nu E_{\sigma s}(X_\om)=\\
=\begin{cases}E_\sigma(X_{\om s})& \text{if }~~~ \ell(\om s)=\ell(\om)+1\\ \\
\delta\big(\nu_s,h\big)\delta\big(\nu_s^{-1},h\big)\cdot
E_\sigma(X_{\om s})&\text{if }~~~\ell(\om s)=\ell(\om)-1.\end{cases}
\label{BS-locE}\end{multline}
\item The R-matrix recursion for local elliptic classes is
\begin{multline}
\delta\left(\zeta_s,\om^{-1}(\nu_s)\right) 
\cdot E_\sigma(X_\om) + 
s^\zeta\left(\delta\left(\zeta_s,h\right)\cdot
E_{s\sigma}(X_\om)\right)=\\
=\begin{cases}E_\sigma(X_{s\om})& \text{if }~~~ \ell(s\om)=\ell(\om)+1\\ \\
\delta\big(\om^{-1}(\nu_s),h\big)\delta\big(\om^{-1}(\nu_s^{-1}),h\big)\cdot
E_\sigma(X_{s\om})&\text{if }~~~\ell(s\om)=\ell(\om)-1.\end{cases}\label{Rmx-locE}
\end{multline}
\end{itemize}

\section{Normalization}
Let
\[
\cc(G,\om)=\prod_{s\in \Pp:\; \om(\alpha_s)\in\Phh_+}{\delta(\nu_s^{-1},h)}.
\]
We will work with the following re-scaling of the elliptic Schubert classes
\[
\EE(X_\om)=\cc(G,\om)\cdot E(X_\om),
\]
as well as its {local} version: for $\om,\sigma \in W$ we define 
\[
\EE_{\sigma}(X_\om)=c(G,\om)\cdot E_\sigma(X_\om) =\frac{\EE(X_\om)|_{\sigma}}{eu(T_\sigma G/B)}.
\]
The indexing set in our normalization factor $\cc(G,\om)$     
is in bijection with the set
$$\Ff(G,\om)=\{\alpha_s^\vee\in \Phh^\vee_+: \;\om(\alpha^\vee_s)\in\Phh^\vee_+\}=\Phh^\vee_+\cap \om^{-1}(\Phh^\vee_+),$$
hence we have the alternative expression
\[
\cc(G,\om)=
\prod_{\alpha^\vee\in\Ff(G,\om)}{\delta(h^{-\alpha^\vee},h)}.
\]

The indexing set has the following geometric interpretation: 
Let $\Tf(G,\om)$ be the set of weights appearing in the tangent space of $X_\om$ at the central point. According to  \cite[\S14.12 Theorem (b)]{Borel} 
$$\Tf(G,\om)=\Phh_+\cap \om \Phh_-.$$
Hence
$$\Ff(G,\om)=\Phh_+^\vee\setminus \Tf(G^\vee,\om^{-1})\,.$$
\begin{proposition}\label{normalization_recursion} We have
\begin{equation}\cc(G,\om s)=\begin{cases}\frac1{\delta (\nu_s,h )}\cdot s^\nu\cc(G,\om)&\text{if }~~~\ell(\om s)>\ell(\om),\\ \\
\delta (\nu_s^{-1},h )\cdot s^\nu\cc(G,\om)&\text{if }~~~\ell(\om s)<\ell(\om),
\end{cases}\end{equation}
\begin{equation}\cc(G,s\om)=\begin{cases}\frac1{\delta (\om^{-1}(\nu_s^{-1}),h )}\cdot \cc(G,\om)&\text{if }~~~\ell(s\om )>\ell(\om),\\ \\
\delta (\om^{-1}(\nu_s),h )\cdot  \cc(G,\om)&\text{if }~~~\ell(s\om )<\ell(\om).
\end{cases}\end{equation}
\end{proposition}
\begin{proof} 
We have the following recurrences
\begin{equation}\label{c-BS}\Tf(G^\vee,\om s)=\Tf(G^\vee,\om)\cup\{w(\alpha^\vee_s)\}\qquad\text{if }~~~ \ell(\om s)>\ell(\om),\end{equation}
\begin{equation}\label{c-Rmx}\Tf(G^\vee,s\om)=s^\nu\Tf(G^\vee,\om)\cup\{\alpha^\vee_s\}\qquad\text{if }~~~ \ell(s\om )>\ell(\om).\end{equation}
These can be derived from \eqref{BS-locE} and \eqref{Rmx-locE} or proven directly: \eqref{c-BS} from the inductive description of Bott-Samelson varieties \cite{Ram1}, and \eqref{c-Rmx} using induction as in the proof of \cite[Prop. 7.3]{RW19}.

From \eqref{c-BS} we obtain
\begin{equation*}\Ff(G,s\om)=\Ff(G,\om)\setminus\{w^{-1}(\alpha^\vee_s)\}\qquad\text{if }~~~ \ell(s\om )>\ell(\om).\end{equation*}
Using the fact that $\Phh_+^\vee\setminus \{\alpha^\vee_s\}=s^\nu(\Phh_+^\vee)\setminus \{-\alpha^\vee_s\}$ (see \cite[Proposition of \S1.4]{Humphreys}) we have
\begin{equation*}\Ff(G,\om s)=\Phh_+\setminus \left(s^\nu\Tf(G^\vee,\om)\cup \{\alpha^\vee_s\}\right)= s^\nu(\Phh_+)\setminus \{-\alpha^\vee_s\}\setminus s^\nu\Tf(G,\om ),\end{equation*}
and hence
\begin{equation}\label{BS-set2}\Ff(G,\om s)=s^\nu \Ff(G,\om)\setminus \{-\alpha^\vee_s\}
\qquad\text{if }~~~ \ell(\om s)>\ell(\om).\end{equation}
Now let $\ell(\om s)<\ell(\om)$, and apply \eqref{BS-set2} with $\om:=\om s$. We get
$\Ff(G,\om )=s^\nu \Ff(G,\om s)\setminus \{-\alpha^\vee_s\}$,
and hence 
\begin{equation*}\Ff(G,\om s)=s^\nu \Ff(G,\om )\cup s^\nu \{-\alpha^\vee_s\}=s^\nu \Ff(G,\om )\cup \{\alpha^\vee_s\}.\end{equation*}
The proof for the $\ell(s\om )<\ell(\om)$ case is similar.
\end{proof}
\begin{remark}\rm
The normalizing factor can also be presented as
\begin{equation}\label{f-interpretation}\cc(G,\om)=\ii\,E_{\om^{-1}\tao}(X^\vee_{\om^{-1}\tao})\,,\end{equation}
where $\ii$ is the inversion of $\T^\vee$ variables, or in other words, the variable $\zetab_s=\nu_s^{-1}$ is substituted by $\nu_s$. 
Equation \eqref{f-interpretation} follows since
$$E_{\sigma}(X^\vee_{\sigma})=\prod_{s\in \Pp: \;\alpha_s\in \sigma(\Phh_-)}{\delta(\e^{\alpha^\vee_s},h)}=\prod_{s\in \Pp: \;\alpha_s\in \sigma(\Phh_-)}{\delta(\zetab^{-1}_s,h)}.$$
We take $\sigma=\om^{-1}\tao$ and use the equalities  $\tao(\Phh_+)=\Phh_-$:
$$E_{\om^{-1}\tao}(X^\vee_{\om^{-1}\tao})=\prod_{s\in \Pp: \;\alpha_s\in \om^{-1}\tao(\Phh_-)}{\delta(\zetab^{-1}_s,h)}=
\prod_{s\in \Pp: \;\om(\alpha_s)\in \Phh_+}{\delta(\zetab^{-1}_s,h)}\,.
$$
Applying inversion of the variables we obtain equation \eqref{f-interpretation}.
\end{remark}

\begin{proposition} \label{prop:EE_recursions}
Let $s\in S$ be a simple reflection. The modified elliptic class $\EE$ satisfies the recursions
\begin{equation}
\EE_{\sigma}(X_{\om  s}) =
\frac{\delta( \sigma(\zeta_s),\nu_s)}
{\delta(\nu_s,h)}\cdot 
s^\nu\EE_{\sigma}(X_{\om })  
+ 
\frac{\delta(\sigma(\zeta_s),h)}{\delta(\nu_s,h)}\cdot
 s^\nu\EE_{\sigma s}(X_{\om }),
\label{B-S-loc-cor}
\end{equation}
\begin{equation}
\EE_{\sigma}(X_{s\om}) =
\frac{\delta(\zeta_s,\om^{-1}(\nu_s))}
{\delta(\om^{-1}(\nu_s^{-1}),h)} \cdot
\EE_{\sigma}(X_{\om }) 
+
 \frac{\delta(\zeta^{-1}_s,h)}
{\delta(\om^{-1}(\nu_s^{-1}),h)} \cdot s^\zeta\EE_{s\sigma}(X_{\om}),
\label{Rmx-loc-cor}
\end{equation}
as well as 
\[
\EE_{\tau}(X_{\id})=\begin{cases}
\prod_{s\in \Pp} \delta(\nu_s^{-1},h) & \tau=\id \\
0  & \tau\not=\id.
\end{cases}
\]
\end{proposition}

\begin{proof} 
The proof is a direct consequence of \eqref{BS-locE}, \eqref{Rmx-locE}, and Proposition \ref{normalization_recursion}.
\end{proof}

\noindent Observe the favorable feature of the normalised local class $\EE_\sigma(X_\om)$ that its recursion does not have to be split into two cases.

\begin{example}\label{SL2ex} \rm 
For $G=\SL_2$ we have $W=\{\id,\tau\}$ and
$$\begin{matrix} 
\EE_{\id}(X_{\id})  =\delta\big(\frac{\mu_1}{\mu_2},h\big),&
\EE_{\tau}(X_{\id})=0,\\
\EE_{\id}(X_\tau)  =\delta\big(\frac{z_2}{z_1},\frac{\mu_2}{\mu_1}\big),&
\EE_{\tau}(X_\tau)=\delta\big(\frac{z_1}{z_2},h\big).\end{matrix}$$
The elliptic classes for $G=\SL_3$ and $\Sp(2)$ are given in \cite[\S12]{RW19}, and some examples  for $\SO(5)$ are given in Section \ref{sec:last} below.
\end{example}

\begin{remark}\rm The normalization 
\[
\Em_\sigma(X_\om)=\prod_{\alpha^\vee\in\Ff(G,\om)}{\delta(h^{-\alpha^\vee},h)^{-1}}\cdot E_\sigma(X_\om)=
\prod_{s\in \Pp}{\delta(\nu_s^{-1},h)^{-1}}\cdot \EE_\sigma(X_\om)
\]
used in \cite{RW19} also has nice behavior for recursions, but for the purpose of our duality theorem (Theorem~\ref{main} below) the normalization $\EE_\sigma(X_\om)$ is necesary.
\end{remark} 
\section{The main theorem}

We are ready to state the coincidence between local elliptic classes on $G/B$ and $G^{\vee}/B^{\vee}$. 

\begin{theorem}\label{main} 
The modified elliptic class $\EE$ satisfies the duality
$$
(-1)^{\ell(\tao)}\,\Big(\EE_{\tao \om^{-1}}(X_{\tao\sigma^{-1}})\Big)^\#
=
\EE_{\sigma}(X^\vee_\om)\,.
$$
where $\#$ denotes precomposing with
$$\T^\vee\times \T\times\CC^*\to \T\times \T^\vee\times\CC^*,$$
$$(a,b,c)\mapsto (\tao(b)^{-1},a,c^{-1}),$$
which leads to the substitution
$$
\zeta_s:=\tao(\nub_s)\,,\qquad \nu_s:=\zetab_s^{-1}\,,\qquad h:= h^{-1}.
$$
\end{theorem}
Note that
for a positive root $\alpha_s:\tt\to \CC$ the composition $\tao(\alpha_s)=\alpha_s\circ \tao$ is a negative root. Similarly $\tao(\alpha_s^\vee)=\alpha_s^\vee\circ \tao$ is a negative coroot. Hence 
$$\tao(\nub_s)=h^{\tao(\alpha_s^\vee)}=h^{-\alpha_\s^\vee}=\nub_\s^{-1}\,,$$
where $$\s=\tao s\tao$$
is the conjugate reflection. This is a simple reflection by \cite[\S1.8]{Humphreys}.
\medskip

In the case of $G=\SL_n$, when we identify $G$ with $G^\vee$, the substitution takes the form
$z_k:= \mu_{n+1-k}$, $\mu_k:= z_k^{-1}$, $h:= h^{-1}$. Hence, Theorem \ref{main} for $G=\SL_n$ takes the form
\begin{equation}\label{AWwants}
\EE_\sigma(X_\om)(z_1,\ldots,z_n;\mu_1,\ldots,\mu_n;h)
= 
\EE_{\tao\om^{-1}}(X_{\tao\sigma^{-1}})\left( \mu_{n+1},\ldots,\mu_1;z_1^{-1},\ldots,z_n^{-1};h^{-1}\right).
\end{equation}

\begin{example} \rm
For $G=\SL_2$ Theorem \ref{main} expresses the fact that the functions in Example \ref{SL2ex} satisfy
\[
\begin{tabular}{rl}
$-\EE_\tau(X_\tau)|_{z_1:=\mu_2,\, z_2:=\mu_1,\, \mu_1:=z_1^{-1}\!,\, \mu_2:=z_2^{-1}, \, h:=h^{-1}}=$ & 
$\EE_{\id}(X_{\id})$,
\\
$-\EE_\id(X_\tau)|_{z_1:=\mu_2,\, z_2:=\mu_1,\, \mu_1:=z_1^{-1},\, \mu_2:=z_2^{-1},\,  h:=h^{-1}}=$&
$\EE_{\id}(X_{\tau})$, \\
$-\EE_\id(X_\id)|_{z_1:=\mu_2,\, z_2:=\mu_1,\, \mu_1:=z_1^{-1},\, \mu_2:=z_2^{-1},\,  h:=h^{-1}}=$&
$\EE_{\tau}(X_{\tau})$.
\end{tabular}
\]
\end{example}

\begin{remark} \rm
The fact that $\omega$ and $\sigma$ play opposite roles on the two sides of the duality is a prediction from physics; stemming from the comparison of so-called vertex functions on the two sides of 3d mirror symmetry, and the {\em switch} between a pair of defining differential equations for vertex functions, see more details in \cite[\S 1.1]{RSVZ0}. The concrete form of the precomposition (ie. the reverse order of variables, and the inverses) is explained in type A by regarding $T^*G/B$ as a ``Cherkis bow variety,'' and noting that its 3d mirror dual {\em as a bow variety} is, although isomorphic with $T^*G/B$, the isomorphism is through some reparametrizations, see \cite[\S 5.4]{RS}. On the other hand we have to admit, that except the $A$ case the duality was
observed by us without any prediction coming from physics.
For a general group $G$ the cotangent space $T^*G/B$ is not a
bow variety in the sense of \cite{RS}.

\end{remark}

\begin{remark}\rm
The main theorem of \cite{RSVZ} is the $G=\SL_n$ special case of our Theorem \ref{main}, that is, it is equivalent to \eqref{AWwants}. Namely, in that paper an elliptic characteristic class $W_\om$, named {\em elliptic stable envelope},  is associated to the permutation $\om$. The restriction of $W_{\om}$ to the fixed point $\sigma\in S_n$ is denoted by $A_{\om,\sigma}$. The duality proved in that paper is \cite[Theorem 4]{RSVZ}:
\begin{equation}\label{RSVZmain}
A_{\om,\sigma}(z_1,\ldots,z_n;\mu_1,\ldots,\mu_n;h)=(-1)^{\binom{n}{2}} A_{\tao\sigma^{-1},\tao\om^{-1}}(\mu_n,\ldots,\mu_1;z_1^{-1},\ldots,z_n^{-1},h).
\end{equation}
In fact, the class $W_\om$ and our $\EE(X_\om)$ are proportional by
\begin{equation}
\EE(X_\om)\cdot 
eu(TG/B) \cdot 
\mathop{\prod_{i<j}}_{\om_i<\om_j} \vt\left(h\frac{\mu_j}{\mu_i}\right) \cdot
\mathop{\prod_{i<j}}_{\om_i>\om_j} \left( \vt\left(\frac{\mu_j}{\mu_i}\right) \frac{\vt(h)}{\vt'(1)}\right) =
W_\om\cdot
\mathop{\prod_{i<j}}_{\om_i<\om_j} \delta\left( \frac{\mu_i}{\mu_j},h\right).
\end{equation}
Using this proportionality, and the fact that 
\begin{equation}\label{hnoh}
\EE_\sigma(X_\om)=
\EE_\sigma(X_\om)|_{h\to h^{-1}} \cdot
\mathop{\prod_{i<j}}_{\om_i<\om_j} \frac{\vt\left(h\frac{\mu_i}{\mu_j}\right)}{\vt\left(h\frac{\mu_j}{\mu_i}\right)} \cdot
\mathop{\prod_{i<j}}_{\sigma_i>\sigma_j} \frac{ \vt\left(h\frac{z_{\sigma(j)}}{z_{\sigma(i)}}\right)}{ \vt\left(h\frac{z_{\sigma(i)}}{z_{\sigma(j)}}\right)},
\end{equation}
straightforward calculation yields that \eqref{RSVZmain} is equivalent to \eqref{AWwants}.
\end{remark}

\medskip

\noindent{\em Proof of Theorem \ref{main}.} 
We reorganize \eqref{Rmx-loc-cor} to get

\begin{equation*}
s^\zeta\EE_{s\sigma}(X_{\om})
=-
\frac{\delta(\zeta_s,\om^{-1}(\nu_s))}
{\delta(\zeta^{-1}_s,h)} \cdot
\EE_{\sigma}(X_{\om }) 
+ 
\frac{\delta(\om^{-1}(\nu_s^{-1}),h)}{\delta(\zeta^{-1}_s,h)}
 \cdot
\EE_{\sigma}(X_{s\om}) 
.
\end{equation*}
Applying $s^\zeta$ yields
\begin{equation*}
\EE_{s\sigma}(X_{\om})
=-
\frac{\delta(\zeta_s^{-1},\om^{-1}(\nu_s))}
{\delta(\zeta_s,h)} \cdot
s^\zeta\EE_{\sigma}(X_{\om }) 
+
\frac{\delta(\om^{-1}(\nu_s^{-1}),h)}{\delta(\zeta_s,h)}
 \cdot
s^\zeta\EE_{\sigma}(X_{s\om}) 
.
\end{equation*}
Now we substitute $\om:=\tao\sigma^{-1}$ and $\sigma:=\tao \om^{-1}$ and obtain
\begin{equation*}
\EE_{s\tao \om^{-1}}(X_{\tao\sigma^{-1}})
=-
\frac{\delta(\zeta_s^{-1},(\sigma\tao(\nu_s))}
{\delta(\zeta_s,h)} \cdot
s^\zeta\EE_{\tao \om^{-1}}(X_{\om }) 
+
\frac{\delta((\sigma\tao(\nu_s^{-1}),h)}{\delta(\zeta_s,h)}
 \cdot
s^\zeta\EE_{\tao \om^{-1}}(X_{s\tao\sigma^{-1}}) 
.
\end{equation*}
In the variables with bars this translates to
\begin{equation*}
\EE_{s\tao \om^{-1}}(X_{\tao\sigma^{-1}})
=-
\frac{\delta(\nub_s,(\sigma\tao(\zetab^{-1}_s))}
{\delta(\nub_s^{-1},h)} \cdot
s^\nub\EE_{\tao \om^{-1}}(X_{\om }) 
+
\frac{\delta((\sigma\tao(\zetab_s),h)}{\delta(\nub_s^{-1},h)}
 \cdot
s^\nub\EE_{\tao \om^{-1}}(X_{s\tao\sigma^{-1}}) 
.
\end{equation*}
Since $\tao(\zetab_s)=\zetab_\s^{-1}$:
\begin{equation*}
\EE_{\tao(w\s)^{-1}}(X_{\tao\sigma^{-1}})
=-
\frac{\delta(\nub_s,\sigma(\zetab_\s))}
{\delta(\nub_s^{-1},h)} \cdot
s^\nub\EE_{\tao \om^{-1}}(X_{\tao\sigma^{-1}}) 
+
\frac{\delta(\sigma(\zetab_\s^{-1}),h)}{\delta(\nub_s^{-1},h)}
 \cdot
s^\nub\EE_{\tao \om^{-1}}(X_{\tao(\sigma \s)^{-1}}) 
.
\end{equation*}
Exchanging $s$ with $\s$ gives
\begin{equation*}
\EE_{\tao(ws)^{-1}}(X_{\tao\sigma^{-1}})
=-
\frac{\delta(\nub_\s,\sigma(\zetab_s))}
{\delta(\nub_\s^{-1},h)} \cdot
s^\nub\EE_{\tao \om^{-1}}(X_{\tao\sigma^{-1}}) 
+
\frac{\delta(\sigma(\zetab_s^{-1}),h)}{\delta(\nub_\s^{-1},h)}
 \cdot
s^\nub\EE_{\tao \om^{-1}}(X_{\tao(\sigma s)^{-1}}) 
.
\end{equation*}
Now we apply the operation $\#$, that is the substitution   $h:= h^{-1}$ and $\nub_\s:= \nub_s=\tao(\nub_\s)^{-1}$, and obtain
\begin{equation*}
\EE_{\tao(ws)^{-1}}(X_{\tao\sigma^{-1}})^\#
=-
\frac{\delta(\nub_s,\sigma(\zetab_s))}
{\delta(\nub_s^{-1},h^{-1})} \cdot
s^\nub\EE_{\tao \om^{-1}}(X_{\tao\sigma^{-1}})^\# 
+
\frac{\delta(\sigma(\zetab_s^{-1}),h^{-1})}{\delta(\nub_s^{-1},h^{-1})}
 \cdot
s^\nub\EE	_{\tao \om^{-1}}(X_{\tao(\sigma s)^{-1}})^\# 
.
\end{equation*}
Since $\delta(1/a,1/b)=-\delta(a,b)=-\delta(b,a)$ we have
\begin{equation*}
\EE_{\tao(ws)^{-1}}(X_{\tao\sigma^{-1}})^\#
=
\frac{\delta(\sigma(\zetab_s),\nub_s)}
{\delta(\nub_s,h)} \cdot
s^\nub\EE_{\tao \om^{-1}}(X_{\tao\sigma^{-1}})^\# 
+
\frac{\delta(\sigma(\zetab_s),h)}{\delta(\nub_s^,h)}
 \cdot
s^\nub\EE_{\tao \om^{-1}}(X_{\tao(\sigma s)^{-1}})^\# .
\end{equation*}
This is exactly recursion \eqref{B-S-loc-cor}  for $G^\vee$. 
The initial condition is the of the form
$$\EE_{\tao\id}(X_{\tao \sigma})^\#=
\begin{cases} 
\EE_{\tao}(X_{\tao})^\#&\text{ if }\sigma=\id\\
0&\text{ if }\sigma\neq \id.
\end{cases}$$
Note that
\begin{multline*}\EE_{\tao}(X_{\tao})^\#=\prod_{s\in \Pp}\delta(\zeta_s^{-1},h)^\#=\prod_{s\in \Pp}\delta(\nub_\s,h^{-1})=\prod_{s\in \Pp}\delta(\nub_s,h^{-1})=\\
(-1)^{|\Pp|}\prod_{s\in \Pp}\delta(\nub_s^{-1},h)=(-1)^{|\Pp|}\EE_{\id}(X^\vee_{\id})\,.\end{multline*}
The conclusion follows.
$\square$

\begin{remark} \rm
The  substitution $\T\times \T^\vee\times\CC^*\to \T^\vee\times \T\times\CC^*$, $(a,b,h)\mapsto (\tao(b)^{-1},a,h^{-1})$, denoted by $\#$, is a key component of the theorem. Denoting it by $\#_G$ we see that 
\[
\#_{G^\vee} \circ \#_G : (a,b,h)\mapsto (\tao(a)^{-1}, \tao(b)^{-1}, h)\,,
\]
is not the identity in general and only $(\#_{G^\vee} \circ \#_G)^2=\id$. 
It would be interesting to know the physical interpretation of this phenomenon in the terminology of Section \ref{sec:physics}.  \end{remark}

\begin{remark}\rm
Composing the statement of Theorem \ref{main} for $G$ and that for $G^\vee$ results the constraint
\[
\EE_\sigma(X_\om)=\left( \EE_{\tau_0 \sigma \tau_0}(X_{\tau_0 \om \tau_0})\right)^{\#_{G^\vee}\circ \#_G}
\]
of local elliptic classes, where the composition $\#_{G^\vee} \circ \#_G$ takes the form
$$(\zeta_s,\nu_s)\;\mapsto (\zeta_{\tao s\tao},\nu_{\tao s\tao}).$$
This constraint is not trivial in type $A_n$ ($n>1$), $D_n$ ($n$ odd) and $E_6$ only; in other types $\tao$ is central, in fact $\tao=-\id\in W \subset \GL(\tt^*)$, see \cite[Exercise 4.10]{BjoBr}.
\end{remark}

\begin{remark}\rm 
The push-forward to the point of the elliptic class is called the elliptic genus. The elliptic genus of a smooth variety has an interpretation as the Euler characteristic of the Elliptic complex, \cite[\S2.1]{BoLi1}. For a singular variety or a singular pair such an interpretation is not clear. Just like the stringy numbers defined by Batyrev \cite{Ba} the elliptic genus is a combination of some invariants of a resolution which miraculously does not depend on the resolution. The local elliptic class (when we assume the torus action) in its full generality is hard to decipher. Only in the case of the quotient singularities for a finite groups a partial interpretation is available. When $q\to 0$ we obtain the motivic Chern class, which can be understood as a graded dimension of invariants of some representation. Precisely, it is given by the Molien series, \cite{Donten-BuryWe}. For an attempt of an interpretation of the full elliptic class see \cite[\S5]{MikoszWe}. The understanding of the numerical properties of the local elliptic class is only satisfactory for quotient singularities. For Schubert varieties in $G/B$ it would be interesting to give a translation into terms of representation theory of $G$. A relation with Verma modules in prime characteristic was given in \cite{SZZ},  this is done only after the specializtion to motivic Chern classes.
\end{remark}

\section{Example: the dual pair $\SO(5)$ and $\Sp(2)$} \label{sec:last}

Consider the Langlands dual pair $G=\SO(5)$, $G^\vee=\Sp(2)$. Their Weyl group is the dihedral group of order 8, generated by the reflections 
$s_1=\begin{pmatrix} 0 & 1 \\ 1 & 0\end{pmatrix}$, $s_2=\begin{pmatrix} 1 & 0 \\ 0 & -1\end{pmatrix}$ acting on the rank 2 weight lattice. The simple roots of $\SO(5)$ are listed in \cite[Chapter 6, Planche II]{Bou81}
$$\alpha_1=(1,-1)\,,\qquad \alpha_2=(0,1)\,,$$
$$\alpha^\vee_1=(1,-1)\,,\qquad \alpha^\vee_2=(0,2)\,.$$
Let $z_1, z_2$ be the coordinates on $\T$ and $\mu_1^{-1},\mu_2^{-1}$ the coordinates on $\T^\vee$ (we use inverses to agree with the convention originating from \cite{RTV}). Then
$$\zeta_1=\tfrac{z_2}{z_1}\,,\qquad \zeta_2=\tfrac1{z_2}\,,$$
$$\nu_1=\tfrac{\mu_2}{\mu_1}\,,\qquad \nu_2=\tfrac1{\mu_2^2}\,.$$
For the dual group $G^\vee=\Sp(2)$ we have (\cite[Chapter 6, Planche III]{Bou81})
$$\zetab_1=\tfrac{\z_2}{\z_1}\,,\qquad \zetab_2=\tfrac1{\z_2^2}\,,$$
$$\nub_1=\tfrac{\mub_2}{\mub_1}\,,\qquad \nub_2=\tfrac1{\mub_2}\,.$$
The two tables below display some local elliptic classes of $\EE_\sigma(X_\om)$ for these two groups, as can be derived from the recursions in Proposition \ref{prop:EE_recursions}. For brevity we write $(a|b)$ for $\delta(a,b)$.
\begin{equation}\tag{$\SO(5)$}
\begin{tabular}{|p{1.5cm}|p{2.5cm}|p{2.5cm}|p{2.5cm}|p{2.5cm}|p{1.5cm}|}  \hline
 ${}_\om\backslash {}^\sigma$ & $1$ & $s_1$ & $s_2$ & $s_1s_2$  & \ldots  \\
 \hline
$1$ & 
   \tiny $(\mu_1^2|h)\,(\frac{\mu_1}{\mu_2}|h)\,\times$ $(\mu_1\mu_2|h)\,(\mu_2^2|h)$ &
   \tiny $0$ & 
   \tiny $0$ & 
   \tiny $0$ & 
   \ldots \\  \hline
$s_1$ & 
   \tiny $(\mu_1^2|h)\,(\mu_1\mu_2|h)\,\times$ $(\mu_2^2|h)\,(\frac{z_2}{z_1}|\frac{\mu_2}{\mu_1})$&  
   \tiny $(\mu_1^2|h)\,({\mu_1}{\mu_2}|h)\,\times$ $(\mu_2^2|h)\,(\frac{z_1}{z_2}|h)$ &
   \tiny  $0$& 
   \tiny $0$ & 
   \ldots \\  \hline
$s_2$ & 
   \tiny $(\mu_1^2|h)\,(\frac{\mu_1}{\mu_2}|h)\,\times$ $(\mu_1\mu_2|h)\,(\frac{1}{z_2}|\frac{1}{\mu_2^2})$ & 
   \tiny $0$ & 
   \tiny $(\mu_1^2|h)\,(\frac{\mu_1}{\mu_2}|h)\,\times$ $(\mu_1\mu_2|h)\,(z_2|h)$ &
   \tiny $0$ & 
   \ldots \\  \hline
$s_1s_2$ & 
   \tiny $ (\mu_1^2|h)\,(\frac{\mu_1}{\mu_2}|h)\,\times$ $(\frac{1}{z_2}|\frac{1}{\mu_2^2})\,(\frac{z_2}{z_1}|\frac{1}{\mu_1\mu_2})$ & 
   \tiny $ (\mu_1^2|h)\,(\frac{\mu_1}{\mu_2}|h)\,\times$ $(\frac{1}{z_1}| \frac{1}{\mu_2^2})\,(\frac{z_1}{z_2}|h)$ &
   \tiny $ (\mu_1^2|h)\,(\frac{\mu_1}{\mu_2}|h)\,\times$ $(z_2|h)\,(\frac{z_2}{z_1}|\frac{1}{\mu_1\mu_2})$ &
   \tiny $ (\mu_1^2|h)\,(\frac{\mu_1}{\mu_2}|h)\,\times$ $(z_1|h)\,(\frac{z_1}{z_2}|h)$ 
   & \ldots \\  \hline
$\vdots$ & $\vdots$ & $\vdots$ & $\vdots$ & $\vdots$ & $\ddots$\\\hline
\end{tabular}
\end{equation}

\begin{equation}\tag{$\Sp(2)$}
\begin{tabular}{|p{1.5cm}|p{1.5cm}|p{2.5cm}|p{2.5cm}|p{2.5cm}|p{2.5cm}|}  \hline
 ${}_\om\backslash {}^\sigma$  
& \ldots & $s_1s_2$ & $s_1s_2s_1$ & $s_2s_1s_2$ & $\tao$   \\
 \hline
$\vdots$ & $\ddots$ & $\vdots$ & $\vdots$ & $\vdots$ & $\vdots$ \\
\hline
$s_1s_2$ & 
   \ldots & 
   \tiny $
 ({\mub_1}|h)\, ( \tfrac{{\mub_1}}{{\mub_2}}|h )\,\times$ $({\z_1}^2|h )\, ( \tfrac{{\z_1}}{{\z_2}}|h )$ &
   \tiny $0$ & 
   \tiny $0$ & 
   \tiny $0$  \\  \hline
$s_1s_2s_1$ & 
   \ldots & 
   \tiny $ ({\mub_2}|h)\,({\z_1}^2|h)\,\times$ $(\tfrac{1}{{\z_1}{\z_2}}| \tfrac{{\mub_2}}{{\mub_1}})\,(\tfrac{{\z_1}}{{\z_2}}|h )$ &  
   \tiny $({\mub_2}|h)\,({\z_1}^2|h )\,\times$ $(\tfrac{{\z_1}}{{\z_2}}|h)\,({\z_1} {\z_2}|h)$ &
   \tiny  $0$& 
   \tiny $0$ \\  \hline
$s_2s_1s_2$ & 
   \ldots &
   \tiny $(\tfrac{{\mub_1}}{{\mub_2}}|h)\,({\z_1}^2|h)\,\times$ $(\tfrac{1}{{\z_2}^2}|\tfrac{1}{{\mub_1}})\,
    (\tfrac{{\z_1}}{{\z_2}}|h )$ & 
   \tiny $0$ & 
   \tiny $(\tfrac{{\mub_1}}{{\mub_2}}|h )\,({\z_1}^2|h)\,\times$ $({\z_1}{\z_2}|h)\,({\z_2}^2|h)$ &
   \tiny $0$ \\  \hline
$\tao$ & 
   \ldots & 
   \tiny $(\z_1^2| h)\,  (\tfrac{\z_1}{\z_2}| h)\,\times$
$(\tfrac1{\z_2^2}|\tfrac1{\mub_2})\, (\tfrac1{\z_1\z_2}|\tfrac{\mub_2}{\mub_1})$ & 
   \tiny $({\z_1}^2|h)\,(\tfrac{1}{{\z_2}^2}|\tfrac{1}{{\mub_2}})\,\times$ $(\tfrac{{\z_1}}{{\z_2}}|h)\,({\z_1}{\z_2}|h)$ &
   \tiny $({\z_1}^2|h )\,( \tfrac{{\z_2}}{{\z_1}}| \tfrac{{\mub_2}}{{\mub_1}} )\,\times$
    $({\z_1} {\z_2}|h)\, ({\z_2}^2|h )$ &
   \tiny $({\z_1}^2|h )\, ( \tfrac{{\z_1}}{{\z_2}}|h )\,\times$
    $({\z_1} {\z_2}|h)\, ({\z_2}^2|h )$ 
    \\  \hline
\end{tabular}
\end{equation}

Note that the elliptic classes of $X_{\tao}$ and $X^\vee_{\tao}$ can be computed using the reduced words $s_1s_2s_1s_2$ and $s_2s_1s_2s_1$. For example
\begin{align*}
\EE_{s_1s_2}(X_{\tao}^\vee) =&
({\z_1}^2|h ) \,(\tfrac{\z_1}{\z_2}| h)\left(
({\z_1}^2| \tfrac{1}{{\mub_2}} )
( \tfrac{1}{{\z_1} {\z_2}}| \tfrac{1}{{\mub_1}{\mub_2}} )+
( \tfrac{1}{{\z_1}^2}| \tfrac{1}{{\mub_1}} ) ( \tfrac{{\z_1}}{{\z_2}}| \tfrac{1}{{\mub_1} {\mub_2}} )
  +
( \tfrac{1}{{\z_2}^2}| \tfrac{1}{{\mub_1}} )( \tfrac{{\z_2}}{{\z_1}}| \tfrac{{\mub_2}}{{\mub_1}} )
\right)\\
=& (\z_1^2| h)\,(\tfrac{\z_1}{\z_2}| h)\,(\tfrac1{\z_2^2}|\tfrac1{\mub_2}) \,
(\tfrac1{\z_1\z_2}|\tfrac{\mub_2}{\mub_1}).
\end{align*}
In the tables above the first class is given in the form obtained from the word $s_1s_2s_1s_2$ and the second from the word $s_2s_1s_2s_1$.
 The substitution given in Theorem \ref{main} is realized by
$$\mu_i\leftrightarrow \z_i^{-1},\qquad \mub_i\leftrightarrow z_i^{-1},\qquad h\leftrightarrow h^{-1}.$$
We encourage the reader to check that the tables after this substitution match.

\end{document}